\documentclass[11pt]{article}
\usepackage{amsfonts}
\usepackage{amsmath,amsthm,color,epsfig}
\usepackage{amssymb}
\usepackage{amsbsy,mathtools}
\usepackage{enumerate,caption}
\usepackage{bookmark}

\title{Three-dimensional quasi-periodic shifted Green function
  throughout the  spectrum---including Wood anomalies}
\author{Oscar P. Bruno\footnote{Applied and Computational Mathematics,
    Caltech, Pasadena, CA 91125. Email: obruno@caltech.edu},
  \, Stephen P. Shipman\footnote{Dept. of Mathematics, Louisiana State
    University, Baton Rouge, LA \ 70803. Email:
    shipman@math.lsu.edu},
  \, Catalin Turc\footnote{Dept. of Math.
    Sciences, New Jersey Inst. of Technology, Newark, NJ 07102. Email: catalin.c.turc@njit.edu},
    \, Stephanos
  Venakides\footnote{Dept. of Mathematics, Duke University, Durham, NC
    \ 27708. Email: ven@math.duke.edu} }
\newtheorem{Theorem}{Theorem}[section]
\newtheorem{Lemma}[Theorem]{Lemma}

\newcommand{\rr}{\boldsymbol{r}}
\newcommand{\xxi}{\boldsymbol{\xi}}

\newcommand{\bfx}{{\mathbf{x}}}
\newcommand{\tbfx}{\tilde\bfx}
\newcommand{\vv}{\mathbf v}
\newcommand{\dd}{D}

\newcommand{\order}[1]{{\mathcal O}\hspace{-2.5pt}\left(#1\right)}

\newcommand{\ZZ}{\mathbb{Z}}
\newcommand{\RR}{\mathbb{R}}

\newcommand{\coeff}{\frac{i}{2\dd}}

\begin{document}
%\tableofcontents
\newpage
\date{}
\maketitle
%\tableofcontents
\begin{abstract}
  This contribution, Part~II in a two-part series, presents an
  efficient method for evaluation of wave scattering by doubly periodic
  diffraction gratings at or near ``Wood anomaly frequencies".  At
  these frequencies---which depend on the angle of incidence and
  periodicity of the grating, and at which one or more {\em grazing}
  Rayleigh waves exist---the quasi-periodic Green function, structured
  as a doubly infinite lattice sum of translated three-dimensional
  free-space Helmholtz Green functions, ceases to converge.  We
  present a modification of this lattice sum which results by adding
  two types of terms to it. The first type adds linear combinations of
  ``shifted'' Green functions, using shift values that ensure that the
  added spatial singularities introduced by these terms are located
  below the grating and therefore outside of the physical domain. With
  suitable coefficient choices these terms annihilate the growing
  contributions in the original lattice sum and yield algebraic
  convergence. (Convergence of arbitrarily high order can be obtained
  by including sufficiently many shifts.) The second type of added
  terms are quasi-periodic plane wave solutions of the Helmholtz
  equation which reinstate certain necessary grazing modes without
  leading to blow-up at Wood anomalies. In particular, using the new
  quasi-periodic Green function, which we denote by ${G}^q_p(\bfx)$,
  we establish, for the first time, that the Dirichlet problem of
  scattering by a smooth doubly periodic scattering surface at a Wood
  frequency is uniquely solvable.  Additionally, we present an
  efficient high-order numerical method based on the Green function
  ${G}^q_p(\bfx)$ for the problem of scattering by doubly periodic
  three-dimensional surfaces at and around Wood frequencies. We
  believe this is the first solver in existence that is applicable to
  Wood-frequency doubly periodic scattering problems.  We demonstrate the
  proposed approach by means of applications to problems of acoustic
  scattering by doubly periodic gratings at various frequencies,
  including frequencies away from, at, and near Wood anomalies.

\medskip
\indent $\mathbf{Keywords}$: scattering, periodic Green function, lattice sum, smooth truncation, Wood frequency, Wood anomaly, boundary-integral equations, electromagnetic computation
\end{abstract}

%%%%%%%%%%%%%%%%%%%%%%%%%%%%%%%%%%%%%%%%%%%%%%%%%%%%%%%%%%%%%%%%%%%%%%%%%%%%%%%%%%%%

\section{Introduction}
This work presents the second part of a two-part contribution. The
first part~\cite{BSTV1}, which will be referenced as
Part~I throughout this paper, introduced a ``windowed Green function''
method, which, utilizing a smooth cutoff,
approximates the quasi-periodic Green function with
super-algebraically small errors---that is, errors that admit upper
bounds proportional to any negative power of the numbers of terms
used---for configurations that are not close to a certain set of
``Wood frequencies''.  As discussed in Part~I and references therein,
at Wood frequencies the classical quasi-periodic Green function ceases
to exist, and, therefore, integral equation methods based on such
Green functions are inapplicable. Following upon the two-dimensional
work~\cite{BrunoDelourme}, the present Part~II introduces an
additional element in the method: the ``shifted'' quasi-periodic Green
function.  By adding copies of the quasi-periodic Green function that
are shifted perpendicular to the plane of periodicity,
one obtains an algebraic convergence rate that increases with the
number of shifts, even at and around Wood frequencies.  Using the new quasi-periodic Green function ${G}^q_p(\bfx)$, we establish, for the first time, that the Dirichlet
problem of scattering by a smooth doubly periodic surface, or diffraction grating, at a
Wood frequency is uniquely solvable.   We present an
efficient high-order numerical method based on the Green function
${G}^q_p(\bfx)$ for the problem of scattering by doubly periodic
three-dimensional gratings at and around Wood frequencies. We believe
this is the first solver in existence that is applicable to
doubly periodic scattering problems at Wood frequencies.

Wood frequencies depend on the wave vector parallel to the grating, and they
occur when one of the Rayleigh waves, or diffraction orders, is at the transition from propagating to evanescent in the ambient medium and is thus exactly grazing the surface.  Because of this, they are often called ``cutoff frequencies".  Certain scattering anomalies at or near these frequencies, known as ``Wood anomalies", were first observed experimentally by Wood~\cite{Wood} and treated mathematically by Rayleigh~\cite{Rayleigh} and Fano \cite{Fano}.  A brief discussion concerning historical aspects in these regards can be found in~\cite[Remark 2.2]{BrunoDelourme}.

Computation of scattering by
three-dimensional doubly periodic structures at or near Wood frequencies
is particularly challenging. As mentioned above, in boundary integral
methods the classical quasi-periodic Green function ceases to exist at
Wood frequencies, even when the scattering problem admits a unique
solution.  A boundary-integral method that employs the free-space
Green function and enforces periodicity through auxiliary layer
potentials on the boundary of a period has been developed for
two-dimensional problems~\cite{BarnettGreengard,GillmanBarnett}, but a
three-dimensional version of this method does not as yet exist.
Approaches based on finite-element methods~\cite{ChenFriedman} often
rely on the classical quasi-periodic Green functions in order to
enforce the radiation condition at infinity; such approaches must also
necessarily fail at Wood anomalies.  Finally, finite-element methods
exist, such as that presented in the
contribution~\cite{chandler-wilde}, which enforce the radiation
condition on the basis of sponge layers such as the
perfectly-matched-layer technique.  The results of that reference
indicate that such approaches may become problematic for truly
quasi-periodic problems and, we suggest, the difficulties are
compounded at Wood frequencies, at which it would be necessary to damp
waves which travel in directions parallel to the absorbing layer.

The classical quasi-periodic Green function $G^{q}(\bfx)$
($\bfx= (x,y,z)\in \RR^3$) can be constructed as an infinite sum of
translated copies of the free-space Helmholtz Green function
\begin{equation}
G(\bfx) = \frac{e^{ik|\bfx|}}{4\pi|\bfx|}
\end{equation}
with doubly periodically distributed monopole singularities. Indeed,
let $\tilde\bfx=(x,y)$, $\boldsymbol{\alpha}:=(\alpha,\beta)$ (the
Bloch wave-vector) and
\begin{equation}\label{r}
  r_{mn}^2= \left| \bfx + m\vv_1 + n\vv_2 \right|^2  =  \left| \tilde\bfx + m\vv_1 + n\vv_2 \right|^2 +z^2,
\end{equation}
in which $\vv_1$ and $\vv_2$ denote two independent vectors in $\RR^2$
that characterize the periodicity.  The quasi-periodic Green function
can be expressed in the form
\begin{equation}\label{latticesum1}
  G^{q}(\bfx) \,=\,\sum_{m,n\in\ZZ} G(\bfx+ m\vv_1 + n\vv_2) 
  \, e^{-i\boldsymbol{\alpha}\cdot(m \mathbf{v}_1+n\mathbf{v}_2)}\,=\, \frac{1}{4\pi}\sum_{m,n\in\ZZ} \frac{e^{ikr_{mn}}}{r_{mn}}
  \, e^{-i\boldsymbol{\alpha}\cdot\vv_{mn}},
\end{equation}
in which $\vv_{mn} = m \mathbf{v}_1+n\mathbf{v}_2$.
%
%Notice that $(\alpha m+\beta n) =(\alpha\vv_1^*+\beta\vv_2^*)\cdot(m\vv_1+m\vv_2)$.
As is well known, this expansion suffers from notoriously poor
convergence properties.  Various methods to accelerate its
convergence, notably the Ewald
method~\cite{Ewald,Capolino,Papanicolaou}, have been
proposed. Importantly, however, the sum does not converge at Wood
frequencies. This is a difficulty that is not addressed by any of the
aforementioned acceleration approaches.

The classical quasi-periodic Green function additionally admits a
spectral representation that results from application of the
Poisson Summation Formula to the series \eqref{latticesum1}: letting
$\vv_1^*$ and $\vv_2^*$ denote the dual vectors defined by
$\vv_i^*\cdot\vv_j=\delta_{ij}$, and letting $\dd=\|\vv_1\times\vv_2\|$,
$\vv_{j\ell}^* = (2\pi j\ \vv_1^*+2\pi\ell\
\vv_2^*)+\boldsymbol{\alpha}$ and
$\gamma_{j\ell}=(k^2-\| \vv_{j\ell}^* \|^2)^\frac{1}{2}$, we have the
alternative expansion
\begin{equation}\label{Fouriersum}
  G^{q}(\tilde\bfx,z) = \frac{i}{2\dd} \sum_{j,\ell\in\ZZ}
  \frac{1}{\gamma_{j\ell}} e^{i\vv_{j\ell}^*\cdot\tilde\bfx}\, e^{i\gamma_{j\ell}|z|}\,,
\end{equation}
which converges provided $\gamma_{j\ell}\ne 0$ for all
$j,\ell$---that is, away from Wood configurations $(\alpha,\beta,k)$ where one of these exponents vanishes. The branch of the
square root that defines $\gamma_{j\ell}$ is selected in such a way
that $\sqrt{1} =1$, with a branch cut that coincides with the negative
imaginary semiaxis.

The primary and dual periodicity lattices are denoted
by
\begin{equation}\label{Lambda}
  \Lambda = \{\vv_{mn}: m,n\in\ZZ \}\quad\mbox{and}\quad \Lambda^*=\{\vv_{j\ell}^*  : j,\ell\in\ZZ\},
\end{equation}
respectively.  The lattice sum~\eqref{latticesum1} is only defined if
$\gamma_{j\ell}\ne 0$ for all integer pairs $(j,\ell)$.  A Wood frequency (for given $\boldsymbol{\alpha}$) is a value of $k$ for which
at least one of the constants $\gamma_{j\ell}$ vanishes. In such cases
both the spatial
expansion~\eqref{latticesum1} and the spectral
representation~(\ref{Fouriersum}) cease to exist.

This paper presents a Green function method that
enables efficient and accurate evaluation of wave scattering by doubly
periodic structures throughout the spectrum, including frequencies at
and around Wood anomalies.  The present contribution additionally
incorporates the windowing approach introduced in Part~I, which
accelerates the Green-function convergence: a windowed version of the
series~\eqref{latticesum1} converges superalgebraically fast (i.e.,
faster than any power of the window size) for non-Wood configurations.
The convergence of this windowed series deteriorates near Wood
frequencies: the constants in the superalgebraic convergence estimates
established in Part~I grow without bound as Wood configurations are
approached, and the windowed version of the lattice sum
(\ref{latticesum1}) once again fails to converge at Wood frequencies.
When applied to the shifting method introduced in this work, however,
the smooth windowing method increases the algebraic convergence rate
at Wood anomalies by a factor equal to the truncation size raised to
the power $-1/2$.

In order to re-establish convergence at Wood frequencies the proposed
method replaces the free-space Green function term
$ \frac{e^{ikr_{mn}}}{r_{mn}}$ by a $p$-th order equispaced
finite-difference for this function with respect to $z$, with
$p\geq 3$ and with step (or ``shift'') $d >0$. The combined effect of
the windowing and shifting/finite-differencing procedure yields a
Green function which converges rapidly at all frequencies. The
approach is demonstrated in Section~\ref{sec:numerical} via an
application to the problems of sound-soft and sound-hard scattering by
doubly periodic surfaces throughout the spectrum. Other scattering
problems can be treated similarly---as demonstrated in the
contributions~\cite{BrunoDelourme} and~\cite{Bruno_Lado} concerning
two-dimensional diffraction gratings and periodic arrays of cylinders,
respectively.

The remainder of this paper is organized as follows. After some
preliminaries in Section~\ref{prelim}, the shifted
quasi-periodic Green function is introduced in
Section~\ref{shifted-G-function}.  This section contains
a proof of high-order algebraic convergence of the shifted Green
function and an existence and uniqueness proof for
configurations at and around Wood anomalies under Dirichlet boundary
conditions. (Appendix~\ref{sec:integral} contains a lemma used in the
aforementioned uniqueness proof, as well as a simplified existence and
uniqueness proof which is only valid away from Wood anomalies.)
Section~\ref{sec:evaluation} then outlines our numerical
implementation. A variety of results presented in
Section~\ref{sec:numerical} demonstrate the character of the proposed
solvers, for all configurations, far, near and at Wood anomalies.
 
\section{Preliminaries\label{prelim}}

We consider the sound-soft problem of scattering by a doubly periodic
scattering surface 
\[
  \Gamma=\{(\tbfx,z): \tbfx\in\mathbb{R}^2 \mbox{ and } z=f(\tbfx)\}
\]
where $f$ is a smooth doubly periodic function of periodicity $\Lambda$:
\begin{equation}
  f(\tbfx+m\vv_1+n\vv_2)=f(\tbfx)
  \quad \text{for all} \quad  \tbfx\in\mathbb{R}^2\quad \text{and all}\quad m,n\in\ZZ\,.                           \end{equation}
Letting
\begin{equation}\label{om+}
\Omega^+=\{(\tbfx,z): \tbfx\in\mathbb{R}^2 \mbox{ and }
z>f(\tbfx)\},
\end{equation}
and assuming an incident field
\begin{equation}\label{eq:planewave}
  u^{inc}(\mathbf x)=\exp[i(\boldsymbol{\alpha}\cdot\tbfx-\gamma z)]
\end{equation}
impinges upon the surface from above (where
$(\boldsymbol{\alpha},-\gamma)$ is the wavevector and
$|\boldsymbol{\alpha}|^2+\gamma^2=k^2$), the scattered field $u$ under
sound-soft conditions satisfies the equations
\begin{equation}\label{helmholtz}
\begin{cases}
\Delta u+k^2 u=0\ {\rm in}\ \Omega^+\\
u=-u^{inc}\ {\rm on}\ \Gamma,
\end{cases}
\end{equation}
together with the Sommerfeld radiation condition~\cite{Petit}: letting
$z_+ = \max{f}$ we have
\begin{equation}\label{outgoing1}
  u(\tbfx,z) \,=\,  \sum_{j,\ell\in\ZZ}B^+_{j\ell}\, \exp{i[(2\pi j\, \vv_1^*+ 2\pi\ell\, \vv_2^*)+\boldsymbol{\alpha}]\cdot\tilde\bfx}\exp[i\gamma_{j\ell}z], \quad z>z_+.
\end{equation}

Although not physically relevant for the grating problems considered
in this paper, the set
\begin{equation}\label{om-}
\Omega^{-}=\{(\tbfx,z): \tbfx\in\mathbb{R}^2 \mbox{ and
}z<f(\tbfx)\}
\end{equation} 
below $\Gamma$ and the associated radiation condition
\begin{equation}\label{outgoing2}
  u(\tbfx,z) \,=\,  \sum_{j,\ell\in\ZZ}B^-_{j\ell}\, \exp\{i[(2\pi j\, \vv_1^*+ 2\pi\ell\, \vv_2^*)+\boldsymbol{\alpha}]\cdot\tilde\bfx\}\exp[-i\gamma_{j\ell}z], \quad z<z_-.
\end{equation}
($z_- = \min{f}$) will be used in the existence and uniqueness
proofs in Section~\ref{uniqueness} and
Appendix~\ref{sec:integral}.

\section{Shifted quasi-periodic Green function\label{shifted-G-function}}
The finite-difference half-space shifted Green function we use can be
viewed as a generalization of the Dirichlet half-space Green that
results from the method of images.  The
method-of-images Green function decays more rapidly at infinity than
the free-space Green function itself. In the
three-dimensional case under consideration, such decay does not
suffice to induce fast convergence, or even absolute convergence, in a
corresponding series of the form~(\ref{latticesum1}). (In contrast,
absolute convergence does result from use of the method-of-images
Green function in the two-dimensional case~\cite[Lemma 4.3 and
Th.~4.4]{BrunoDelourme}). But viewing the method-of-images Green
function as a finite difference of the lowest order, a generalization
of this idea emerges: as shown in what follows, a higher-order finite
difference of the free-space Green function tends to zero more rapidly
at infinity, and thus gives rise to an absolutely convergent
quasi-periodic Green function series. In fact, the resulting
quasi-periodic Green function can be made to converge with a
prescribed order of accuracy provided finite-differences of
sufficiently high order are used.

To pursue this idea, fix a ``shift'' value $d>0$ and
define the shifted half-space Green function
\begin{equation}\label{Gtildeqp2} 
  \tilde{G}_p(\bfx) = \sum_{q=0}^p a_{pq} G\!\big(\bfx+(0,0,qd)\big) = 
  \sum_{q=0}^p a_{pq} G\!\big(\tbfx+(0,0,z+qd)\big),
\end{equation}
where $a_{pq}$ denote the finite-difference
coefficients~\cite{BrunoDelourme,Knuth}
\begin{equation*}
  a_{pq} = (-1)^q {p\choose q}\,, \qquad 0\leq q\leq p.
\end{equation*}
Clearly, the function $\tilde{G}_p(\bfx)$ has poles at the points
$(0,0,-qd)$ for $0\leq q\leq p$.  As shown in Lemma~\ref{gderivatives}
below, the shifted half-space Green function $\tilde{G}_p(\bfx)$ tends
to zero algebraically fast as $\tbfx$ tends to infinity while $z$
remains bounded. The shifted {\em quasi-periodic} Green function is
then defined by
\begin{equation}\label{Gtildeqp} 
  \tilde{G}^q_p(\bfx) \,= \,\sum_{m,n\in\ZZ} \tilde{G}_p(\bfx + m\mathbf{v}_1+n \mathbf{v}_2)
  \,e^{-i\boldsymbol{\alpha}\cdot(m \mathbf{v}_1+n \mathbf{v}_2)}.
\end{equation}
The function $\tilde{G}^q_p(\bfx)$ has poles at the points
\begin{equation}
  \bfx'_{mnq} = (m\mathbf{v}_1 + n\mathbf{v}_2,\, -qd\,),
  \quad
  m,n,q\in\ZZ,\;\; 0\leq q\leq p\,. 
\end{equation}
Theorem~\ref{thm:algebraic} below shows that the sum on the right-hand
side of equation~\eqref{Gtildeqp} converges algebraically fast: a
truncation of this series to $|m \mathbf{v}_1+n \mathbf{v}_2|\leq A$
results in errors of order $\lceil p/2 \rceil - 1$ (where
$\lceil r \rceil$ denotes the smallest integer greater than or equal
to $r$) for all wavevectors
$\mathbf{k} = (\boldsymbol{\alpha},-\gamma)$ and all periodicity
lattices $\Lambda$, including Wood-anomaly configurations. That
theorem also establishes that somewhat improved accuracies result when
a smooth windowed truncation, as introduced in Part~I, is additionally
employed.

When evaluated at $(\tbfx-\tbfx',z-z')$ with $(\tbfx,z)\in\Gamma$ and
$(\tbfx',z')\in\Gamma$, the terms with $q=1,\dots, p$ in the
series~\eqref{Gtildeqp2}-\eqref{Gtildeqp} are weighted copies of the
free-space Green function with sources at the points $(\tbfx',z'-qd)$
below the grating surface $\Gamma$, while the terms corresponding to
$q=0$ produce the necessary sources on $\Gamma$. It follows that the
layer potentials associated with the shifted Green function
(cf. equation~\eqref{combinedlayer} below) are well defined for all
points $\bfx = (\tbfx,z)$ on and above the grating surface, and
satisfy the Helmholtz equation for $\bfx$ above $\Gamma$.

It is important to note that, as shown in Section~\ref{modified_fnct},
the addition of the shifted terms effectively suppresses all
contributions in equation~\eqref{Fouriersum} that contain vanishing
denominators at Wood frequencies, and thereby reinstates convergence
of the quasi-periodic Green function even at such frequencies. But
such modes are required in the complete solution of the grating
scattering problem. Therefore corresponding quasi-periodic plane-wave
terms need to be added to the Green function to incorporate (now with
controlled coefficients) all the necessary grazing modes, as detailed
in Section~\ref{modified_fnct}.  The corresponding treatments
for simply periodic scattering surfaces and arrays of cylinders in two
dimensions is presented in references~\cite{BrunoDelourme}
and~\cite{Bruno_Lado}, respectively.

\subsection{Convergence of the modified Green function at Wood frequencies}\label{sec:Wood} %%%%%%%%%%%%%%%

Lemma~\ref{gderivatives} below states that, as desired, the shifted
half-space Green functions~\eqref{Gtildeqp2} enjoy enhanced degrees of
decay as $\tbfx\to\infty$. On the basis of this result,
Theorem~\ref{thm:algebraic} then establishes the fast convergence of
two different truncations (using discontinuous and and smooth window
functions) for the spatial lattice sum~\eqref{Gtildeqp}.  In
preparation for the proofs, we examine the individual
terms $\tilde{G}^q_p(\bfx+m\mathbf{v}_1+n \mathbf{v}_2)$. Considering
the relations~\eqref{r},~\eqref{latticesum1} and~\eqref{Gtildeqp2} and
using the notations
\begin{equation}\label{g}
g(\rho,\varepsilon) := \frac{e^{ik\rho\sqrt{1+\varepsilon^2}}}{\rho\sqrt{1+\varepsilon^2}},\quad  \rho_{mn} = | \tbfx+m\vv_1+n\vv_2 |, \quad
  \varepsilon_{mn}=z/\rho_{mn} \quad \mbox{and}\quad \hat\varepsilon_{mn}=d/\rho_{mn},
\end{equation}
% and letting
% \begin{equation}
%   g(\rho,\varepsilon) := \frac{e^{ik\rho\sqrt{1+\varepsilon^2}}}{\rho\sqrt{1+\varepsilon^2}}.
% \end{equation}
the translated Green-function terms in the sum~\eqref{Gtildeqp} can be
expressed in the form
\begin{equation}
\tilde{G}_p(\bfx + m\mathbf{v}_1+n \mathbf{v}_2) =
\sum_{q=0}^p a_{pq} g(\rho_{mn},\varepsilon_{mn}+q\hat\varepsilon_{mn}).
\end{equation}
Equivalently, letting
\begin{equation}\label{h}
  h(\rho,\varepsilon,\hat\varepsilon) :=
  \sum_{q=0}^p a_{pq} g(\rho,\varepsilon+q\hat\varepsilon),
\end{equation}
we have
\begin{equation}
  \tilde{G}_p(\bfx + m\mathbf{v}_1+n \mathbf{v}_2) =h(\rho_{mn},\varepsilon_{mn},\hat\varepsilon_{mn}).
\end{equation}

In order to estimate the asymptotics of the function $h$ as
$\rho\to\infty$ we use the finite-difference relation~\cite[p. 262,
eq. 7]{isaacson:1994}
\begin{equation}\label{pthdifference}
  \sum_{q=0}^p a_{pq} f(\varepsilon+q\hat\varepsilon)) \,=\,
  (-1)^p{\hat\varepsilon}^p f^{(p)}(\varepsilon+\xi)
  \qquad\xi \in[0, p\hat\varepsilon ],
\end{equation}
which is valid for every $p$-times continuously differentiable
function $f$. Thus, for each pair of values of $\rho$ and
$\hat\varepsilon$ there exists
$\xi_{\rho,\hat\varepsilon}\in[0, p\hat\varepsilon ]$ such that
\begin{equation*}
  h(\rho,\varepsilon,\hat\varepsilon)
  =(-1)^p{\hat\varepsilon}^p \frac{\partial^p g}{\partial\varepsilon^p}(\rho,\varepsilon+\xi_{\rho,\hat\varepsilon}),
\end{equation*}
and, therefore
\begin{equation}\label{hbound0}
  h\left(\rho,\,z/\rho,\,d/\rho\right) =  (-1)^p\left(\frac{d}{\rho}\right)^p\, \frac{\partial^p g}{\partial\varepsilon^p} (\rho,z/\rho+\xi_{\rho,d/\rho}),\quad
    \xi_{\rho,d/\rho}\in[0, pd/\rho ].
\end{equation}
In view of the asymptotic bounds on $\partial^p g / \partial
\varepsilon^p$ in Lemma~\ref{gderivatives} below, one obtains $
h_p\big(\rho,\,z/\rho,\,d/\rho\big) \,=\, \order{1/\rho^{\lceil
    \frac{p}{2} \rceil+1}} $ as $ (\rho\to\infty) $ where the constant
in the $\cal{O}$-term, which depends on $z$, $d$ and $p$, can be taken
to be fixed if $d$ and $p$ are given and $z$ is contained in a bounded
subset of $\RR$.

\begin{Lemma}\label{gderivatives}  %%%%%%%%%  LEMMA  %%%%%%%%%%%%%%%%%%%
  The $p$-th order derivative of the function $g$ with respect to
  $\varepsilon$ satisfies
\begin{equation}\label{gbound}
  \left | \frac{1}{\rho^p} \frac{\partial^p g}{\partial\varepsilon^p}(\rho,\varepsilon)\right| \leq  \frac{C}{\rho^{\lceil \frac{p}{2} \rceil+1}} \qquad (\rho\geq 1,\  |\varepsilon | <
1),
\end{equation}
where the constant $C$ is independent of $\rho$ and $\varepsilon$ for
all $\rho\geq 1$ and all $\varepsilon$ satisfying $|\varepsilon| <
1$. Here, for real $x$, $\lceil x\rceil$ denotes the smallest integer
larger than or equal to $x$.  The estimate~\eqref{gbound} is sharp:
the left-hand side in that equation does not decay like $1/\rho^t$ for
any $t>\lceil \frac{p}{2} \rceil+1$.  Furthermore,
\begin{equation}\label{hbound}
  h\big(\rho,\,z/\rho,\,d/\rho\big) \,=\, \order{\frac{1}{\rho^{\lceil \frac{p}{2} \rceil+1}}} \qquad (\rho\to\infty).
\end{equation}
\end{Lemma}

\noindent {\em\bfseries Proof.} It suffices to establish
that~\eqref{gbound} holds and is sharp; Equation (\ref{hbound}) then
follows directly from (\ref{hbound0}) and (\ref{gbound}).  In order to
obtain the relation~\eqref{gbound} we first note that the
$p^\text{th}$ derivative $\frac{\partial^p g}{\partial\varepsilon^p}$
of~$g$ with respect to $\varepsilon$ can be expressed as a finite
linear combination of the form
\begin{equation}\label{termofdg}
  \frac{\partial^p g}{\partial\varepsilon^p}(\rho,\varepsilon) = \sum_{(m,n,\ell)\in S_p} C_{m,n,\ell}^p \frac{\varepsilon^m(ik\rho)^n}{\sqrt{1+\varepsilon^2\,}^\ell}\, \frac{e^{ik\rho\sqrt{1+\varepsilon^2}}}{\rho},
\end{equation}
in which $S_p$ is a certain set of triples of non-negative integer
indices and $C_{m,n,\ell}^p$ denote real valued coefficients.
Defining
\begin{equation}\label{T_def}
  T_{m,n,\ell}(\rho,\varepsilon) = 
  \frac{\varepsilon^m(ik\rho)^n}{\sqrt{1+\varepsilon^2\,}^\ell} \quad \text{for } m,n,\ell\geq0, 
\end{equation}
 we may thus write
\begin{equation}\label{p_exp}
  \frac{\partial^p g}{\partial\varepsilon^p}(\rho,\varepsilon) = \sum_{(m,n,\ell)\in S_p} C_{m,n,\ell}^p T_{m,n,\ell}(\rho,\varepsilon) \frac{e^{ik\rho\sqrt{1+\varepsilon^2}}}{\rho}.
\end{equation}
But, it is easy to check that 
\begin{equation}\label{recur0}
\displaystyle g(\rho,\varepsilon) = T_{0,0,1}(\rho,\varepsilon) \frac{e^{ik\rho\sqrt{1+\varepsilon^2}}}{\rho},
\end{equation}
and that, defining the unary operators
\begin{align}
   & f^-T_{m,n,\ell}(\rho,\varepsilon) = T_{m+1,n,\ell+2}(\rho,\varepsilon), \\
   & f^0T_{m,n,\ell}(\rho,\varepsilon) = T_{m+1,n+1,\ell+1}(\rho,\varepsilon),\quad\mbox{and} \\
   & f^+T_{m,n,\ell}(\rho,\varepsilon) = \begin{cases}
    T_{m-1,n,\ell}(\rho,\varepsilon)\quad &m\geq 1\\
    0\quad &m = 0
\end{cases}
\label{fplus}
\end{align}
we have
\begin{equation}\label{recur}
  \displaystyle \frac{\partial}{\partial\varepsilon} \left(
    T_{m,n,\ell}(\rho,\varepsilon)\frac{e^{ik\rho\sqrt{1+\varepsilon^2}}}{\rho} \right) =
  \left( -\ell f^-T_{m,n,\ell} + f^0T_{m,n,\ell} + mf^+T_{m,n,\ell}
  \right) \frac{e^{ik\rho\sqrt{1+\varepsilon^2}}}{\rho} \quad(m\geq0).
\end{equation}
We now may (and do) redefine the set $S_p$ to ensure that it contains
exactly the triples $(m,n,\ell)$ corresponding to a sequence of $p$
applications of the unary operators~\eqref{fplus}. Thus calling
$F=\{f^-,f^0,f^+\}$ we let
\begin{equation}\label{Sp_let}
  S_p = \{(m,n,\ell):T_{m,n,\ell} = f_1f_2\dots f_p T_{0,0,1} \mbox{
    where } f_j\in F \mbox{ for } j=1,\dots,p\}.
\end{equation}

Application of the operator $f^+$ results in a decrease by
one in the power of $\varepsilon$ in the expression~\eqref{T_def} for
$T_{m,n,\ell}(\rho,\varepsilon)$, while application of either $f^0$
or $f^-$ results in an increase by one in that power. Susbstituting
$\varepsilon=z/\rho$ in~\eqref{T_def}, on the other hand, we
obtain
\begin{equation*}
  T_{m,n,\ell}(\rho,z/\rho) \sim z^m(ik)^n \rho^{n-m} \quad (\rho\to\infty) \quad\text{for }m\geq0,
\end{equation*}
and we thus define the $\rho$-order $Q$ of $T_{m,n,\ell}$ by
\begin{equation*}
  Q(T_{m,n,\ell}) = n-m.
\end{equation*}
Notice that an application of the operator $f^+$ (resp. $f^0$, $f^-$)
to $T_{m,n,\ell}$ results in an increase (resp. no change, decrease)
in the $\rho$-order $Q$.

Let $(m_0,n_0,\ell_0)$ be such that $Q_+=Q(T_{m_0,n_0,\ell_0})$ is greater
than or equal to $Q(T_{m,n,\ell})$ for all $(m,n,\ell)\in S_p$. We
claim that $C_{m_0,n_0,\ell_0}^p\ne 0$. To establish this fact,
first note that, in view of equations~\eqref{termofdg}
and~\eqref{recur}, the coefficient $C_{m,n,\ell}^p$ in the
sum~\eqref{p_exp} equals a sum of several contributions, each one of
which results from a sequence of $p$ operations from the set
$\{f^0,f^+\}$ applied to the root expression
$T_{0,0,1}(\rho,\varepsilon)$. Indeed, while in principle all three
elements in the set $\{f^-,f^0,f^+\}$ appear as a contribution to a
coefficent $C_{m,n,\ell}^p$, it is easy to check that $f^-$ cannot
appear as a contribution towards the maximum order coefficient
$C_{m_0,n_0,\ell_0}$, since, as pointed out above, the $f^-$ operator
decreases the $\rho$-order $Q$.

Since none of the $f^0$ and $f^+$ contributions are negative, it only
remains to check that there is at least one positive contribution to
the coefficient of $C(m_0,n_0,\ell_0)$ of $T(m_0,n_0,\ell_0)$ of
$\rho$-order $Q_+$. But, in view of equation~\eqref{fplus}, a nonzero
contribution of the form $f_1f_2\dots f_p T_{0,0,1}$ to the
coefficient $C(m_0,n_0,\ell_0)$ can only result provided no more than
half of $p$ operators used equals $f^+$---which implies, in
particular, that $Q_+\leq p/2$.  In fact we have $Q_+=p/2$
(resp. $Q_+=(p-1)/2$) for $p$ even (resp. for $p$ odd), and a positive
contribution to $C(m_0,n_0,\ell_0)$ is provided by
$T(m_0,n_0,\ell_0)=(f^+)^\frac{p}{2}(f^0)^\frac{p}{2}T(0,0,1)$ for $p$
even and by
$T(m_0,n_0,\ell_0)=(f^+)^\frac{p-1}{2}(f^0)^\frac{p+1}{2}T(0,0,1)$ for
$p$ odd. Taking into account the factor of $1/\rho$ in each one of the
terms in equation~\eqref{p_exp}, equation~\eqref{gbound} follows and,
thus, in view of~\eqref{hbound0}, so does~\eqref{hbound}. The proof is
now complete.  \hfill $\blacksquare$

\medskip

Now we are able to prove the algebraic convergence of the lattice sum for $\tilde{G}^q_p(\bfx)$.

%%%%%%%%%%%%%%%% THEOREM -- REFLECTED GREEN FUNCTION  alg convergence

\begin{Theorem}[Modified Green function for all frequencies; algebraic convergence]\label{thm:algebraic}
  Let $\chi(r)$ be a smooth truncation function equal to $1$ for
  $r<r_1$ and equal to $0$ for $r>r_2$ ($0<r_1<r_2$), and let $p$
  denote an integer such that $p\geq 3$.  Then, for all real triples
  $(k,\alpha,\beta)=(k,\boldsymbol{\alpha})$ ($k\not=0$) the sums
\begin{eqnarray}
\label{GreenReflectedCutoff1}
  G^{p,A}(\tbfx,z) &=& \frac{1}{4\pi}\!\!
  \sum_{\stackrel{m,n\in\ZZ}{|m\vv_1+n\vv_2|\leq A}}
  \!\!e^{-i\boldsymbol{\alpha}\cdot(m \mathbf{v}_1+n \mathbf{v}_2)} \sum_{q=0}^p a_{pq} \frac{e^{ikr_{mn}^q}}{r_{mn}^q}\quad\mbox{and}\\
\label{GreenReflectedCutoff2}
  \hat G^{p,A}(\tbfx,z) &=& \frac{1}{4\pi}\,  \sum_{m,n\in\ZZ} e^{-i\boldsymbol{\alpha}\cdot(m \mathbf{v}_1+n \mathbf{v}_2)} \sum_{q=0}^p a_{pq} \frac{e^{ikr_{mn}^q}}{r_{mn}^q} \,
  \chi(\tilde r_{mn}/A) \,,
\end{eqnarray}
where $(r_{mn}^q)^2 = | \tbfx + m\vv_1+n\vv_2 |^2 + (z+qd)^2$\, and\,
$\tilde r_{mn} = | \tbfx + m\vv_1+n\vv_2 |$\,, converge to a radiating
quasi-periodic modified Green function $\tilde{G}^q_p(\bfx)$ which
satisfies the Partial Differential Equation
\begin{equation}\label{GqperHelmholtz}
  \nabla^2 \tilde{G}^q_p(\bfx) + k^2 \tilde{G}^q_p(\bfx) \,=\,
        -\sum_{m,n\in\ZZ}\sum_{q=0}^p\, e^{i\boldsymbol{\alpha}\cdot(m \mathbf{v}_1+n \mathbf{v}_2)} \delta(\bfx-\bfx'_{mnq})\,,
\end{equation}
as well as the quasi-periodicity condition\; $\tilde{G}^q_p(\bfx+(m\vv_2+n\vv_2,0)) =
\tilde{G}^q_p(\bfx)\,e^{i\boldsymbol{\alpha}\cdot(m \mathbf{v}_1+n \mathbf{v}_2)}$.\; Further, there
exists a constant $C_p=C_p(k,\alpha,\beta)$ for which
\begin{align}
  &\left| G^{p,A}(\bfx) - \tilde{G}^q_p(\bfx) \right| < \frac{C_p}{A^{\lceil p/2 \rceil - 1}}\quad\mbox{and} \label{estimate_1}\\
  &\left| \hat G^{p,A}(\bfx) - \tilde{G}^q_p(\bfx) \right| < \frac{C_p}{A^{\lceil p/2 \rceil - 1/2}} \label{estimate_2}
\end{align}
for all sufficiently large values of $A$.
\end{Theorem}

\noindent {\em\bfseries
  Proof.}
The sum \eqref{GreenReflectedCutoff1} can be re-expressed in the form
\begin{equation}\label{Gpak}
  G^{p,A}(\bfx) \,=\, \frac{1}{4\pi}
  \sum_{\stackrel{m,n\in\ZZ}{|m\vv_1+n\vv_2|\leq A}}
    h(\rho_{mn}, z/\rho_{mn}, d/\rho_{mn})
    e^{-i(\alpha m + \beta n)}.
\end{equation}
But, in view of \eqref{hbound} and letting $\nu = \lceil p/2 \rceil +
1$ we see that
\begin{equation*}
  \frac{1}{4\pi} h\left(\rho_{mn}, z/\rho_{mn}, d/\rho_{mn}\right)
  = \frac{1}{\left|\vv_{mn}\right|^\nu} H(\bfx;m,n)
\end{equation*}
for some function $H$ which, for certain constants $C$ and $M$
satisfies $|H(\bfx;m,n)|<C$ as long as $|\vv_{mn}|>M$.  Thus, the sum
in~\eqref{Gpak} converges to a limit $\tilde{G}^q_p(\bfx)$ as $A\to
+\infty$, and for $A>M$ we have
\begin{equation}
  \left| G^{p,A}(\bfx) - \tilde{G}^q_p(\bfx) \right|
  \,\leq \sum_{\stackrel{m,n\in\ZZ}{|m\vv_1+n\vv_2|> A}} \frac{C}{\left| \vv_{mn} \right|^\nu} 
  \,<\, \frac{C_p}{A^{\nu-2}}
\end{equation}
---a relation which establishes the desired result~\eqref{estimate_1}.

The proof of the bound~\eqref{estimate_2} follows in part the proof
of Theorem~2.1 in Part~I~\cite{BSTV1}.  For simplicity we assume $\tbfx=(x,y) = 0$
and $\boldsymbol{\alpha} = 0$; the extension of
the proof to nonzero values of these quantities is handled easily via
consideration of standard properties of the Fourier transform, as
described explicitly at the end of the proof in~\cite{BSTV1}.  Let $U$ denote the finite set
\begin{equation}\label{U_def}
  U = \left\{ j\vv^*_1+\ell\vv^*_2 : \gamma_{j\ell}=0 \right\}\subseteq \Lambda^*.
\end{equation}
This set is nonempty exactly when $k$ is a Wood frequency.  
Since $k\not=0$, one has $(0,0)\not\in U$.  The assumption $p\geq3$ implies
$\nu\geq3$.

Let $z_q = z+dq$.  We may
re-express~\eqref{GreenReflectedCutoff2} in the form
\begin{equation}
  4\pi \, \hat G^{p,A}(0,0,z) = \sum_{\rr\in\Lambda} \chi(\rr/A)
    \sum_{q=0}^p \frac{\exp\big(ik\sqrt{|\rr|^2+z_q^2\,}\,\big)}{\sqrt{|\rr|^2+z_q^2\,}}\,.
\end{equation}
Using the Poisson Summation Formula, this sum is transformed into a
lattice sum in the Fourier variable:
\begin{equation}
  4\pi \, \hat G^{p,A}(0,0,z)
  = \sum_{\xxi\in\Lambda^*}
  \mathcal{F} \left[  \chi(\rr/A)
    \sum_{q=0}^p \frac{\exp(ik\sqrt{|\rr|^2+z_q^2\,}\,)}{\sqrt{|\rr|^2+z_q^2\,}} \right] (\xxi) = S_1 + S_2
  \end{equation}
in which
\begin{align}
  & S_1= \sum_{\xxi\in\Lambda^*\setminus U} \sum_{q=0}^p
  \int_{\RR^2} \chi(\rr/A) \frac{\exp(ik\sqrt{|\rr|^2+z_q^2\,}\,)}{\sqrt{|\rr|^2+z_q^2\,}} e^{2\pi i\xxi\cdot\rr} d\rr \quad \mbox{and} \\
  & S_2 = \sum_{\xxi\in U} \int_{\RR^2} \chi(\rr/A) e^{2\pi
    i\xxi\cdot\rr} \left( \sum_{q=0}^p
  \frac{\exp(ik\sqrt{|\rr|^2+z_q^2\,}\,)}{\sqrt{|\rr|^2+z_q^2\,}}\right)
  d\rr.
\end{align}

The proof of in~\cite[Theorem~2.1]{BSTV1} establishes that the sum
$S_1$ ($U$ is empty in that work) converges superalgebraically to a
limit $L$, that is, it is equal to $L + \order{A^{-n}}$ for each
  positive integer $n$. Thus, it only remains to show that the sum
  $S_2$ converges, with the difference from its limit being of order $\order{A^{-(\nu-3/2)}}$
  (or, equivalently, $\order{A^{-(\lceil p/2 \rceil - 1/2)}}$) as
  $A\to +\infty$.  Each fraction in $S_2$ can be expressed as an
  exponential in $r=|\rr|$, multiplied by a Laurent expansion,
\begin{equation}
  \frac{\exp(ik\sqrt{r^2+z_q^2})}{\sqrt{r^2+z_q^2}}
  = \frac{e^{ikr}}{r}g_q(r)\, , \ \ \ \ \ \ \ \ \  g_q(r) = 1 + \sum_{j=1}^\infty a^q_j r^{-j}\,.
\end{equation}%

The coefficients $a^q_j$ depend on $z$, and the expansion is convergent when $r>|z_q|$.
In view of (\ref{h}) and (\ref{hbound}) we see that
\begin{equation}
  \sum_{q=0}^p g_q(r) = \frac{g(r)}{r^{\nu-1}}\quad\mbox{where}
  \quad g(r) = \sum_{j=0}^\infty a_j r^{-j}\,
  \quad\mbox{for}
  \quad r>z_q\quad(\mbox{with $a_0\not=0$}),
\end{equation}
and clearly
\begin{equation}
  g(Ar) \to a_0 \;\text{ as }\; A\to\infty,
\end{equation}
with uniform convergence over the set $r\geq r_1$.

In order to study the contribution by the sum $S_2$ we define the
polar coordinates $\rr=(r,\theta)$ and we note that, since $\xxi\in
U$, we must necessarily have $\xxi=(k/2\pi,\gamma)$. Then, using the
rescaling $\rho=r/A$ and the notation $\psi(\rho)=1-\chi(\rho)$, and
in view of the fact that $\nu\geq 3$, we obtain
\begin{multline}
  \int_{\RR^2} \chi(\rr/A) e^{2\pi i\xxi\cdot\rr} \left(\sum_{q=0}^p \frac{\exp(ik\sqrt{|\rr|^2+z_q^2\,}\,)}{\sqrt{|\rr|^2+z_q^2\,}}\right) d\rr\\
=  \int_{\RR^2}e^{2\pi i\xxi\cdot\rr} \left(\sum_{q=0}^p \frac{\exp(ik\sqrt{|\rr|^2+z_q^2\,}\,)}{\sqrt{|\rr|^2+z_q^2\,}}\right) d\rr
   -  \int_0^{2\pi}\hspace{-6pt}\int_0^\infty \psi(A^{-1} r) e^{ik r(\cos\theta+1)} \frac{g(r)}{r^{\nu-1}} dr d\theta \\
  = \tilde L \,-\, A^{2-\nu} \int_{r_1}^\infty \psi(\rho) \frac{g(A\rho)}{\rho^{\nu-1}} \left(\int_0^{2\pi} e^{iAk(\cos\theta+1)\rho} d\theta \right)d\rho\\
 = \tilde L \,-2\pi\, A^{2-\nu} \int_{r_1}^\infty \psi(\rho) \frac{g(A\rho)}{\rho^{\nu-1}} e^{iAk\rho}J_0(Ak\rho)d\rho,\label{fivesix}
\end{multline}
where $J_0$ denotes the Bessel function of order $0$. Clearly, the
number $\tilde L$ does not depend on $A$. In view of the well known
Bessel function asymptotics $J_0(x)\sim (2/x\pi)^{1/2}\cos(x-\pi/4),\ x\to +\infty$, and since $k\not=0$, we obtain the relation
\begin{equation}
  \int_{\RR^2} \chi(\rr/A) e^{2\pi i\xxi\cdot\rr} \left(\sum_{q=0}^p \frac{\exp(ik\sqrt{|\rr|^2+z_q^2\,}\,)}{\sqrt{|\rr|^2+z_q^2\,}}\right) d\rr
  \,=\, \tilde L \,+\, \order{A^{3/2-\nu}}\,. 
\end{equation}
It follows that the sum $S_2$ converges with an error of
$\order{1/A^{\lceil p/2 \rceil - 1/2}}$ as $A\to\infty$ (since $\nu =
\lceil p/2 \rceil + 1$).  Together with the superalgebraic convergence
of $S_1$ as $A\to\infty$, this fact
establishes~\eqref{estimate_2}. The proof is now complete. \hfill
$\blacksquare$

%%%%%%%%%%%%%%%%%%%%%%%%%%%%%%%%%%%%%%%%%%

\subsection{Complete Green function in Fourier space\label{modified_fnct}}

In view of the Fourier expression~\eqref{Fouriersum} for the Green
function away from Wood anomalies we obtain the corresponding expression
\begin{equation*}
  \tilde{G}^q_p(\tbfx,z)
  = \frac{i}{2\dd} \sum_{j,\ell\in\ZZ} \frac{1}{\gamma_{j\ell}}e^{i\,\vv_{j\ell}^*\cdot\tbfx}\,
  \sum_{q=0}^p a_{pq} e^{i\gamma_{j\ell}|z+qd|}\,.
\end{equation*}
for the shifted Green function. For $z>0$ this expression can be made
to read
\begin{equation}\label{z_pos}
  \tilde{G}^q_p(\tbfx,z)
  = \frac{i}{2\dd} \sum_{j,\ell\in\ZZ} \frac{1}{\gamma_{j\ell}}e^{i\,\vv_{j\ell}^*\cdot\tbfx}\,e^{i\gamma_{j\ell}z}(1-e^{i\gamma_{j\ell}d})^p.
\end{equation}
In view of the limit 
$\lim_{\gamma_{j\ell}\to 0} \frac{(1-e^{i\gamma_{j\ell}d})^p}{\gamma_{j\ell}}=0$
for $p\geq 2$, we see that
that~\eqref{z_pos} can be evaluated even at Wood anomalies: letting
\begin{equation}\label{U_def2}
  U =\{(j,\ell)\in\ZZ^2\, |\, \gamma_{j\ell}=0\},
\end{equation}
we may continuously extend the function $\tilde{G}^q_p(\tbfx,z)$ to
all frequencies, including Wood configurations, by means of the
expression
\begin{equation}\label{z_pos_2}
  \tilde{G}^q_p(\tbfx,z)
  = \frac{i}{2\dd} \sum_{(j,\ell)\not \in U} e^{i\,\vv_{j\ell}^*\cdot\tbfx}\,e^{i\gamma_{j\ell}z}\frac{(1-e^{i\gamma_{j\ell}d})^p}{\gamma_{j\ell}}.
\end{equation}
Unfortunately, however, if $\gamma_{j\ell}=0$ for some $(j,\ell)$, the
corresponding Fourier component is not present in the Green
function~\eqref{z_pos_2} and, therefore, use of this Green function
cannot give rise to a uniquely solvable system of integral equations.
To tackle this difficulty we follow~\cite{BrunoDelourme} and introduce
a modified version $G^q_p$ of $\tilde{G}^{q}_p$, that is outgoing for
$z\to\infty$ (but not for $z\to-\infty$) and which contains all
necessary Fourier harmonics, even at Wood frequencies. The modified
Green function is given~by
%Suppose now that, at $(k,\alpha,\beta)=(k_0,\alpha_0,\beta_0)$, $\gamma_{j\ell}$ vanish or nearly vanish for some pairs $(j,\ell)$.
%
\begin{equation*}
  {G}^q_p(\bfx) \;\;=\;\; \tilde{G}^q_p(\bfx)
  \;+\; v(\bfx),
\end{equation*}
where $v(\tbfx,z)$ denotes a solution of the homogeneous Helmholtz
equation of the form
\begin{equation}\label{hom_helm}
  v(\tbfx,z) = \frac{i}{2\dd} \sum_{(j,\ell)\in U} b_{j\ell}\, e^{i\vv^*_{j\ell}\cdot\tbfx} e^{i\gamma_{j\ell}\, z},
\end{equation}
where $b_{j\ell}\ne 0$ are arbitrary non-zero complex
constants. The function ${G}^q_p(\tbfx,z)$ is
$\boldsymbol{\alpha}$-quasi-periodic in $\tbfx$ with periods $\vv_1$
and $\vv_2$, it satisfies~(\ref{GqperHelmholtz}) and, crucially, it
contains all Fourier harmonics.

As shown in Section~\ref{uniqueness} and demonstrated numerically in
Section~\ref{sec:numerical}, the complete Green function ${G}^q_p$ can
be used to obtain uniquely-solvable integral-equation formulations
around Wood frequencies. The analysis presented in
Section~\ref{uniqueness} relies, in part, on use of yet another Green
function, namely, a {\em non-radiating} Green function defined by a
{\em slowly-convergent} series which, however, 1)~is well defined at
Wood frequencies; and 2)~unlike the shifted Green function
${G}^q_p$, is a Helmholtz solution for $z<0$, and is therefore
well suited for use as part of a proof that concerns the PDE domain
and its complement at a Wood frequency. This rather peculiar Green
function is introduced in the following section, and it is then used
in the uniqueness proof presented in Section~\ref{uniqueness}.

\subsection{ All space (non-radiating) Green function at and around Wood
  frequencies \label{non-rad}}
In view of the relation
\begin{equation}\label{zexp}
  \frac{e^{i\gamma_{j\ell}|z|}}{\gamma_{j\ell}} 
  = \frac{\cos(\gamma_{j\ell}\, z)}{\gamma_{j\ell}} + i\,\frac{\sin(\gamma_{j\ell}|z|)}{\gamma_{j\ell}},
\end{equation}
each term in the classical quasi-periodic Green
function~\eqref{Fouriersum} equals the sum of two quantities, the
first of which diverges and the second of which tends to $i|z|$ as
$\gamma_{j\ell}\to 0$.  In view of the relation~\eqref{z_pos_2}, both
terms can be made to vanish by using a $p$-th order finite-difference
of shifted Green functions. As an alternative, a direct removal of the
diverging term (which amounts to addition of a solution of Helmholtz'
equation) does produce a Helmholtz Green function for
$(k,\alpha,\beta)$ in a neighborhood of a given Wood anomaly triple
$(k_0,\alpha_0,\beta_0)$ at which $\gamma_{j\ell} = 0$. The resulting
Green function at the Wood configuration $(k_0,\alpha_0,\beta_0)$ is
thus given by
\begin{equation}\label{B_Green}
  B^q(\tbfx,z) :=\; \coeff \sum_{(j,\ell)\not\in U} e^{{i\vv_{j\ell}^*\cdot\tbfx}}
  \frac{1}{\gamma_{j\ell}} e^{i\gamma_{j\ell}|z|}
  \,+\, \coeff \sum_{(j,\ell)\in U} e^{{i\vv_{j\ell}^*\cdot\tbfx}}\; i |z|.
\end{equation}
This is not an outgoing Green function, on account of the terms that
contain~$|z|$ as a factor. But as indicated in the previous section,
the function $B^q(\tbfx,z)$ is a solution of the Helmholtz equation
except at the periodically distributed singular points
$(m\vv_1 + n\vv_2,0)$, and, unlike the shifted Green function
${G}^q_p$, it does not contain any additional singularities.

\subsection{Uniquely solvable integral equations around Wood frequencies\label{uniqueness}}

We seek a scattered field above $\Gamma$ in the form of a combined
single- and double-layer potential
\begin{equation}\label{combinedlayer}
  u(\bfx) = \int_{\Gamma^{per}} \left(i\eta {G}^q_p(\mathbf x-\mathbf{x}')+\xi\frac{\partial {G}^q_p(\mathbf x-\mathbf{x}')}{\partial{n}(\mathbf x')}\right)
  \phi(\mathbf{x}')ds(\bfx')  
\end{equation}
in terms of a quasi-periodic density $\phi$ defined on $\Gamma^{per}$,
where
\begin{equation}
  \Gamma^{per} = \left\{ (\tbfx,z) : \tbfx = a\vv_1+b\vv_2 \text{ with } 0\leq \left\{ a,b \right\} \leq 1,
    z = f(\tbfx) \right\}.
\end{equation}
The domain $\Gamma^{per}$ is that part of $\Gamma$ that lies above the
unit cell
\begin{equation}\label{Q}
  Q = \left\{ \tbfx = a\vv_1+b\vv_2 : 0\leq \left\{ a,b \right\} < 1 \right\},
\end{equation}
of the periodic lattice~$\Lambda$.

The function $u$ defined in~\eqref{combinedlayer} is quasi-periodic
and outgoing as $z\to\infty$, and it satisfies the Helmholtz equation
in $\RR^3\setminus\bigcup_{q=0}^p(\Gamma^{per}-(0,0,qd))$. This
function is a solution of~\eqref{helmholtz} if and only if $\phi$
solves the integral equation
\begin{equation}\label{eq:SLfEPer}
\frac{\xi\phi(\mathbf x)}{2}+\int_{\Gamma^{per}}\left(i\eta {G}^q_p(\mathbf x-\mathbf{x}')+\xi\frac{\partial {G}^q_p(\mathbf x-\mathbf{x}')}{\partial{n}(\mathbf x')}\right)\phi(\mathbf{x}')ds(\mathbf{x}')
=-e^{i(\boldsymbol{\alpha}\cdot\tbfx-\gamma z)},\quad \bfx\in\Gamma^{per}.
\end{equation}
The following theorem establishes that this integral equation is
uniquely solvable as long as the shift distance $d>0$ is selected in
such a way that the restrictions
\begin{equation}\label{consts}
\left(1-e^{i\gamma_{j\ell}d}\right)\ne 0\quad \mbox{for all}\quad (j,\ell)\in\ZZ^2
\end{equation}
are satisfied. It is important to note that, since $\gamma_{j\ell}$ is
an imaginary quantity for $j,\ell$ large enough,
equation~\eqref{consts} amounts to a finite number of
constraints. % Additionally, it
% should be noted that 2)~An alternative approach leading to a complete
% Green function (i.e., one that contains non-zero Fourier components of
% order $(j,\ell)$ for all $(j,\ell)\in\ZZ^2$) can alternatively be
% achieved by using a correction of the form~\eqref{hom_helm} but with
% indexes beyond the Wood anomaly set $U$, thereby allowing for wider
% ranges of values of the parameter $d$. For brevity the alternative
% approach~2) will not be considered in what follows, but the
% corresponding extension is straightforward.

%%%%%%%%%%%%%%%%%%%%%%% THEOREM %%%%%%%%%%%%%%%%%%%%%%%%%%
\begin{Theorem}\label{thm:uniqueness}
  Let $p\geq 0$, let $\xi\ne 0$ and $\eta\ne 0$ denote real numbers
  satisfying $\eta/\xi<0$, and let $d>0$ be such that~\eqref{consts}
  holds. Then equation~\eqref{eq:SLfEPer} admits a unique solution for
  all triples $(k,\alpha,\beta)$ of wavenumbers, including Wood
  anomalies.
\end{Theorem}

\begin{proof}
  Given that the surface $\Gamma$ is smooth and that $G^{q}_p$ has the
  same singularity as $G$, it follows that the integral operators on
  the left-hand side of equation~\eqref{eq:SLfEPer} are compact
  operators in the space $L^2(\Gamma^{per})$. Thus, by the Fredholm
  theory, the unique solvability of equation~\eqref{eq:SLfEPer} is
  equivalent to the injectivity of the operator on the left-hand side
  of that equation. In order to establish injectivity, and thereby
  complete the proof of the theorem, in what follows we show that any
  solution $\phi\in L^2(\Gamma^{per})$ of the homogeneous equation
\begin{equation}\label{eq:SLfEPer_hom}
  \frac{\xi\phi(\mathbf x)}{2}+\int_{\Gamma^{per}}\left(i\eta {G}^q_p(\mathbf x-\mathbf{x}')+\xi\frac{\partial {G}^q_p(\mathbf x-\mathbf{x}')}{\partial{n}(\mathbf x')}\right)\phi(\mathbf{x}')ds(\mathbf{x}')
  =0,\quad \bfx\in\Gamma^{per},
\end{equation}
must necessarily vanish.

Let $\phi$ satisfy~\eqref{eq:SLfEPer_hom} and let $u$ denote the
corresponding Helmholtz solution~\eqref{combinedlayer}. It follows
that $u$ vanishes on $\Gamma^{per}$. Since $u$ satisfies the outgoing
radiation condition~\eqref{outgoing1} at $+\infty$ (because
${G}^q_p(\bfx)$ does), by uniqueness of solution of the Dirichlet
problem~\eqref{helmholtz}, which holds even at Wood
frequencies~\cite{Petit}, it follows that $u(\bfx)=0$ for all
$\bfx\in\Omega^+$. In particular, letting $u_{j\ell}(z)$ denote the
Fourier coefficients of the doubly periodic function
$u(\tbfx,z)e^{-i\boldsymbol{\alpha}\cdot\tbfx}$ with respect to
$\tbfx$ for $z>z_+$, we have $u_{j\ell}(z)=0$ for all $j,\ell\in\ZZ$.

  Given that for $z>0$
\begin{equation*}
  G^q_p(\tbfx,z)=\frac{i}{2\dd} \sum_{(j,\ell)\not \in U} e^{i\,\vv_{j\ell}^*\cdot\tbfx}\frac{(1-e^{i\gamma_{j\ell}d})^p}{\gamma_{j\ell}}\,e^{i\gamma_{j\ell}z}+\frac{i}{2\dd} \sum_{(j,\ell)\in U} b_{j\ell}\, e^{i\vv^*_{j\ell}\cdot\tbfx} e^{i\gamma_{j\ell}\, z}\,
\end{equation*}
(see Section~\ref{modified_fnct}) it follows that % the Fourier
% expansion of the field $u = u(\bfx)$ is given by
% %
% \begin{equation*}
%   u(\tbfx,z) \,=\, \sum_{j,\ell\in\ZZ} e^{i\vv_{j\ell}^*\cdot\tbfx}\, u_{j\ell}(z)\,\quad z>z_{+},
% \end{equation*}
% %
% with $z$-dependent coefficients 
% %
\begin{equation}\label{u_jl}
u_{j\ell}(z) =
\begin{cases}
  c_{j\ell}^{+} \frac{(1-e^{i\gamma_{j\ell}d})^p}{\gamma_{j\ell}}
  e^{i\gamma_{j\ell}z},
  &\mbox{for}\quad(j,\ell)\not\in U \\
  c_{j\ell}^{+}b_{j\ell} &\mbox{for}\quad(j,\ell)\in U,
\end{cases}
\end{equation}
in which
\begin{equation}
  c_{j\ell}^{+} = \frac{1}{2\dd}\int_\Gamma \phi(\bfx') e^{-{i\vv_{j\ell}^*\cdot\tbfx'}} e^{-i\gamma_{j\ell}z'}\left[\xi(\vv_{j\ell}^*,\gamma_{j\ell})\cdot \mathbf{n}(\bfx')   -\eta\right]    ds(\bfx').
\end{equation}
Under the assumptions of this theorem, $u$ may vanish only if
$c_{j\ell}^{+}=0$ for all $j,\ell\in\ZZ$.

Using the non-radiating Green function introduced in
Section~\ref{non-rad} we then set
\begin{equation}\label{u0}
  v(\bfx) \,=\, \int_\Gamma 
  \left(
    i\eta\, B^q(\bfx-\bfx') + 
    \xi\, \frac{\partial B^q(\bfx-\bfx')}{\partial{n}(\mathbf x')}
  \right)\phi(\bfx')
  ds(\bfx')\,,\quad \mathbf{x}\not\in\Gamma. 
\end{equation}
In view of~\eqref{B_Green}, for $z>z_+$ the function $v$ admits the
Fourier expansion
\begin{equation*}
  v(\tbfx,z) = \sum_{j,\ell\in\ZZ} v^+_{j\ell}(z)\,e^{{i\vv_{j\ell}^*\cdot\tbfx}} \qquad (z>z_+)\,,
\end{equation*}
in which
\begin{align}\label{v+}
  & v^+_{j\ell}(z) \,=\, c_{j\ell}^{+} \frac{1}{\gamma_{j\ell}} e^{i\gamma_{j\ell}z} &\text{for } (j,\ell) \not\in U, \\
  & v^+_{j\ell}(z) \,=\, i\big(c_{j\ell}^{+} z - c'_{j\ell} \big) +c''_{j\ell} &\text{for } (j,\ell)\in U,
\end{align}
in which (recalling that $|z|=z-z'$ when $\bfx$ is above $\Gamma$, or $z>f(\tbfx)$)
\begin{align}
  &c'_{j\ell} = \frac{1}{2\dd}
    \left[
    -\eta\int_\Gamma\phi(\bfx') \,z'\, e^{-{i\vv_{j\ell}^*\cdot\tbfx'}}ds(\bfx')\; + \;
    \xi\int_\Gamma\phi(\bfx') (\vv_{j\ell}^*,0)\cdot \mathbf{n}(\bfx')\,z'\, e^{-{i\vv_{j\ell}^*\cdot\tbfx'}}ds(\bfx')
    \right],\\
  & c''_{j\ell} = \frac{\xi}{2\dd}\int_\Gamma\phi(\bfx') (0,0,1)\cdot \mathbf{n}(\bfx')\, e^{-{i\vv_{j\ell}^*\cdot\tbfx'}}ds(\bfx')\,.
\end{align}
Given that $c_{j\ell}^{+}=0$ for all $j,\ell\in\mathbb{Z}$, it follows that
\begin{equation*}
  v^+_{j\ell}(z) =
  \renewcommand{\arraystretch}{1.1}
\left\{
  \begin{array}{ll}
     0\,, & (j,\ell)\not\in U\,, \\
     -ic'_{j\ell}+c''_{j\ell}\,, & (j,\ell)\in U\,.
  \end{array}
\right.
\end{equation*}
Therefore
$v(\tbfx,z) \,=\, \sum_{(j,\ell)\in U} (-ic'_{j\ell}+c''_{j\ell})\,
e^{i\vv_{j\ell}^*\cdot\tbfx}$ for $z>z_+$.  The field $v(\bfx)$
satisfies the Helmholtz equation for $\bfx\not\in\Gamma$, and
$v(\bfx)$ is independent of $z$ for $z>z_+$.  In view of the
real-analyticity of $v(\bfx)$ for $\bfx\not\in\Gamma$ and the
uniqueness of analytic continuation, it follows that $v(\bfx)$ is
independent of~$z$ everywhere above $\Gamma$ and we have
\begin{equation}\label{u0above}
  v(\tbfx,z) \,=\, \sum_{(j,\ell)\in U}(-i c'_{j\ell}+c''_{j\ell})\, e^{i\vv_{j\ell}^*\cdot\tbfx} \qquad \text{for } z\geq f(\tbfx)\,.
\end{equation}

Now we turn our attention to the Fourier expansion
\begin{equation*}
  v(\tbfx,z) = \sum_{j,\ell\in\ZZ}v_{j\ell}^-(z) e^{i\vv_{j\ell}^*\cdot\tbfx}
  \quad ,\quad z<z_-
\end{equation*}
of the function $v$ in the region $z<z_-$. For $(j,\ell)\in U$, the
Fourier coefficients in this expansion satisfy
$v_{j\ell}^-(z) = -v_{j\ell}^+$, since the term $|z-z'|$ in
$B^q(\bfx-\bfx')$ equals $-(z-z')$ for $z<z'$. (There is no such
relation for $(j,\ell)\not\in U$.) For $(j,\ell)\in U$ we thus have
\begin{equation*}
  v_{j\ell}^-(z) = ic'_{j\ell} - c''_{j\ell}\,.
\end{equation*}
The function $v(\bfx)$ satisfies the radiation condition~\eqref{outgoing2}
for $z<z_-$ since, in spite of linear terms that are part of the Green
function $B^q$, $v$ itself does not contain such linear
growth for $z<z_-$.

The right-hand side expression in~\eqref{u0above}, considered as a
function defined on all of $\RR^3$, defines a Helmholtz field that
satisfies the outgoing condition for $z\to\infty$ {\itshape and}
$z\to-\infty$.  Thus, by subtracting it from $v(\bfx)$, one obtains
the field
\begin{equation*}
  \tilde v(\tbfx,z) := v(\tbfx,z) - \sum_{(j,\ell)\in U} (-ic'_{j\ell} + c''_{j\ell}) e^{{i\vv_{j\ell}^*\cdot\tbfx}}
  \qquad \text{for } (\tbfx,z)\in\RR^3
\end{equation*}
that satisfies the radiation conditions at both $+\infty$ and
$-\infty$, and that vanishes for $z>f(\tbfx)$.  The jump conditions of
the single- and double-layer potentials that define $v(\bfx)$ in
(\ref{u0}) imply that the limits of $\tilde v(\bfx)$ and its normal
derivative from below $\Gamma$ are
\begin{eqnarray}\label{utilde0limits}
  \tilde v(\tbfx,f(\tbfx)\!-\!0) &=& -\xi\,\phi(\tbfx,f(\tbfx))\,,\\
  \displaystyle{\frac{\partial\tilde v}{\partial{n}(\bfx)}(\tbfx,f(\tbfx)\!-\!0)}
  &=& i\eta\,\phi(\tbfx,f(\tbfx))\,.\label{utilde0limits_2}
\end{eqnarray}
It follows that the function $\tilde v(\bfx)$ restricted to the domain
$\{z\leq f(\tbfx)\}$ satisfies the homogeneous impedance boundary
condition
\begin{equation*}
  \frac{\partial\tilde v}{\partial{n}(\bfx)}(\tbfx,f(\tbfx))+\frac{i\eta}{\xi}\tilde v(\tbfx,f(\tbfx)) =0\quad \mbox{on $\Gamma$}.
\end{equation*}
Since, per the discussion above, $\tilde v$ additionally satisfies the
radiation condition~\eqref{outgoing2}, Lemma~\ref{lemma:uniqueness}
tells us that we must have $\tilde v(\bfx)=0$ for all $\bfx$ below
$\Gamma$.  In view of~\eqref{utilde0limits}
and/or~\eqref{utilde0limits_2} and the assumption $\eta/\xi<0$ it
follows that $\phi(\bfx)=0$ for all $\bfx\in\Gamma$.
\end{proof}

\begin{table}
\begin{center}
\begin{tabular}{|c|c|c|c|c|c|c|c|}
\hline
$k$ & Unknowns & $A$ & $\max|G^A-G^{ref}|$&Iter & $\varepsilon_1$ & $\varepsilon$ \\
\hline
$1$ & $16\ \times\ 16$ &30 & 4.4 $\times$ $10^{-2}$ & 13 & 1.0 $\times$ $10^{-1}$& 1.8 $\times$ $10^{-1}$ \\
$1$ & $16\ \times\ 16$ & 60 & 3.1 $\times$ $10^{-3}$ &13 & 6.3 $\times$ $10^{-3}$& 6.6 $\times$ $10^{-3}$ \\
$1$ & $16\ \times\ 16$ & 120 & 4.3 $\times$ $10^{-4}$ & 13 & 4.8 $\times$ $10^{-4}$& 4.3 $\times$ $10^{-4}$ \\
$1$ & $16\ \times\ 16$ & 160 & 3.7 $\times$ $10^{-5}$ & 13 & 1.3 $\times$ $10^{-4}$& 2.3 $\times$ $10^{-4}$ \\
$1$ & $16\ \times\ 16$ & 240 & 2.4 $\times$ $10^{-6}$ & 13 & 4.9 $\times$ $10^{-6}$& 6.7 $\times$ $10^{-6}$ \\
\hline
\end{tabular}
\caption{\label{errors1} Convergence of the $p=0$ Dirichlet solver
  (unshifted) as $A$ grows, for a configuration away from Wood
  anomalies. Normal incidence ($\boldsymbol{\alpha}=0$) was assumed
  for this example. The reference solution was produced using
  $A=A_{ref}=320$ and $16\ \times\ 16$ unknowns, for which
  $\varepsilon=1.0\times 10^{-6}$.}
\end{center}
\end{table}
% d_radial = radius/32; step_theta= pi/16; fl_radius = 0.2

\begin{table}
\begin{center}
\begin{tabular}{|c|c|c|c|c|c|c|c|c|}
\hline
$k$ & Unknowns & $A$ & \multicolumn{3}{c|}{$G^A$} &  \multicolumn{3}{c|}{$G^{A,p},p=3$}\\
\cline{4-9}
& & & Iter & $\varepsilon_1$ & $\varepsilon$ & Iter &$\varepsilon_1$ & $\varepsilon$ \\
\hline
$6$ & $16\ \times\ 16$ &30 & 16 & 4.9 $\times$ $10^{-3}$& 1.6 $\times$ $10^{-3}$ & 12 & 6.5 $\times$ $10^{-3}$& 1.2 $\times$ $10^{-2}$\\
$6$ & $16\ \times\ 16$ & 60 & 16 & 1.5 $\times$ $10^{-3}$& 4.1 $\times$ $10^{-4}$ & 12 & 3.8 $\times$ $10^{-4}$& 1.5 $\times$ $10^{-5}$\\
$6$ & $16\ \times\ 16$ & 80 & 16 & 2.8 $\times$ $10^{-5}$& 1.1 $\times$ $10^{-5}$ & 12 & 4.7 $\times$ $10^{-6}$& 2.3 $\times$ $10^{-6}$\\
\hline
\end{tabular}
\caption{\label{errors11} Convergence of the $p=3$ Dirichlet solver
  with shift parameter $d=2.4$, away from Wood frequencies, as $A$
  grows. Normal incidence. The reference solution was produced using
  corresponds to $A=A_{ref}=120$, $32\ \times\ 32$ unknowns, for which
  $\varepsilon=4.0\times 10^{-7}$.}
\end{center}
\end{table}

\section{High-order numerical evaluation of the boundary-layer
  potentials with quasi-periodic Green
  functions}\label{sec:evaluation}

For our numerical treatment we reformulate the quasi-periodic
scattering integral equation~\eqref{eq:SLfEPer} in terms of only {\em periodic} functions.
We use the fact that the solution $\phi = \phi^{qper}$ of the integral
equation~\eqref{eq:SLfEPer} is $\boldsymbol{\alpha}$-quasi-periodic with
respect to the lattice $\Lambda$, and that, therefore, the quantity
$$\phi^{per}(\tbfx)\;=\;e^{-i\boldsymbol{\alpha}\cdot\tbfx}\,\phi^{qper}(\tbfx)$$ 
is periodic with respect to $\Lambda$. This allows us to
express the integral equation~\eqref{eq:SLfEPer} in the form
\begin{equation}\label{eq:SLfEPer1}
  \frac{\xi\phi^{per}(\mathbf{x})}{2} + \int_{\Gamma^{per}}\left(\xi\frac{\partial G^{per}_p(\mathbf x-\mathbf{x}')}{\partial{n}(\mathbf x')}+i\eta G^{per}_p(\mathbf x-\mathbf{x}')\right)\phi^{per}(\mathbf{x}')ds(\mathbf{x}') = 
  -e^{-i\gamma f(\tbfx)},\quad \tbfx \in \Gamma^{per},
\end{equation}
where the $\Lambda$-{\em periodic} Green function $G^{per}_p$ is defined by
\begin{equation}\label{eq:Greenp1}
G^{per}_p(\mathbf{x},\mathbf{x}')=G^{q}_p(\mathbf{x},\mathbf{x}')
e^{i\boldsymbol{\alpha}\cdot(\tbfx'-\tbfx)}.
\end{equation}
In practice, for any given $p=0,1,\dots$ the quantity $\hat{G}^{p,A}$
(equation~\eqref{GreenReflectedCutoff2}) is used as an approximation
for $G^{q}_p$ which, when subsituted in~\eqref{eq:Greenp1}, results in
the needed numerical approximation of $G^{per}_p$.  We solve
equation~\eqref{eq:SLfEPer1} by means of the unaccelerated high-order
Nystr\"om procedure introduced in~\cite{BrunoKunyansky} and
Part~I. For all the numerical experiments presented in
Section~\ref{sec:numerical} a single patch was used to represent the
biperiodic surfaces under consideration.

\begin{table}
\begin{center}
\begin{tabular}{|c|c|c|c|c|c|c|}
\hline
$k$ & Unknowns & $A$ & Iter & $\varepsilon_1$ & $\varepsilon$ \\
\hline
$2 \pi$ & $24\ \times\ 24$ & 20 & 19 & 3.2 $\times$ $10^{-2}$& 1.7 $\times$ $10^{-2}$ \\
$2 \pi$ & $24\ \times\ 24$ & 30 & 19 & 2.7 $\times$ $10^{-3}$& 4.7 $\times$ $10^{-3}$ \\
$2 \pi$ & $24\ \times\ 24$ & 40 & 19 & 6.9 $\times$ $10^{-4}$& 4.0 $\times$ $10^{-4}$ \\
$2 \pi$ & $24\ \times\ 24$ & 60 & 19 & ref & 2.4 $\times$ $10^{-6}$ \\
\hline
$2 \pi\pm 10^{-6}$ & $24\ \times\ 24$ & 40 & 19 & 6.8 $\times$ $10^{-4}$& 4.3 $\times$ $10^{-4}$ \\
\hline
\hline
$2\sqrt{2}\pi$ & $24\ \times\ 24$ & 30 & 25 & 1.3 $\times$ $10^{-2}$& 1.5 $\times$ $10^{-2}$ \\
$2\sqrt{2}\pi$ & $24\ \times\ 24$ & 40 & 25 & 6.8 $\times$ $10^{-3}$& 5.5 $\times$ $10^{-3}$ \\
$2\sqrt{2}\pi$ & $24\ \times\ 24$ & 80 & 25 & 8.9 $\times$ $10^{-5}$& 2.1 $\times$ $10^{-4}$ \\
$2\sqrt{2}\pi$ & $24\ \times\ 24$ & 120 & 25 & ref & 3.8 $\times$ $10^{-5}$ \\
% 120, 240, 0.6 pou, 1.6 shift, 16 radial, 12 \theta
\hline
$2\sqrt{2}\pi\pm 10^{-6}$ & $24\ \times\ 24$ & 30 & 25 & 1.6 $\times$ $10^{-2}$& 2.5 $\times$ $10^{-2}$ \\
$2\sqrt{2}\pi\pm 10^{-6}$ & $24\ \times\ 24$ & 80 & 25 & 8.9 $\times$ $10^{-5}$& 1.5 $\times$ $10^{-4}$ \\
\hline
\hline
$4\pi$ & $32\ \times\ 32$ & 30 & 28 & 4.6 $\times$ $10^{-2}$& 4.9 $\times$ $10^{-2}$ \\
$4 \pi$ & $32\ \times\ 32$ & 60 & 28  & 2.4 $\times$ $10^{-3}$& 1.1 $\times$ $10^{-3}$ \\
$4 \pi$ & $32\ \times\ 32$ & 180 & 28  & ref  & 2.2 $\times$ $10^{-4}$ \\
\hline
\end{tabular}
\caption{\label{errors3} Convergence, as $A$ grows, of the $p=3$
  Dirichlet solver at and around Wood frequencies. Shift parameter
  $d=1.4$, GMRES residual tolerance equal to $10^{-6}$.  ``Ref" refers to finely resolved solutions against which the error of the coarser solutions is evaluated.}
\end{center}
\end{table}

\begin{table}
\begin{center}
\begin{tabular}{|c|c|c|c|c|c|c|}
\hline
$k$ & Unknowns & $A$ & Iter & $\varepsilon_1$ & $\varepsilon$ \\
\hline
$2 \pi$ & $24\ \times\ 24$ & 20 & 19 & 5.3 $\times$ $10^{-3}$& 7.4 $\times$ $10^{-3}$ \\
$2 \pi$ & $24\ \times\ 24$ & 30 & 19 & 5.7 $\times$ $10^{-4}$& 1.7 $\times$ $10^{-3}$ \\
$2 \pi$ & $24\ \times\ 24$ & 40 & 19 & 1.1 $\times$ $10^{-4}$& 3.7 $\times$ $10^{-4}$ \\
$2 \pi$ & $24\ \times\ 24$ & 80 & 19 & ref & 4.5 $\times$ $10^{-6}$ \\
\hline
$2 \pi\pm 10^{-6}$ & $24\ \times\ 24$ & 40 & 19 & 1.1 $\times$ $10^{-4}$& 3.6 $\times$ $10^{-4}$ \\
\hline
\hline
$2\sqrt{2}\pi$ & $24\ \times\ 24$ & 30 & 25 & 7.3 $\times$ $10^{-2}$& 5.1 $\times$ $10^{-2}$ \\
$2\sqrt{2}\pi$ & $24\ \times\ 24$ & 40 & 25 & 4.1 $\times$ $10^{-3}$& 2.8 $\times$ $10^{-3}$ \\
$2\sqrt{2}\pi$ & $24\ \times\ 24$ & 80 & 25 & 2.6 $\times$ $10^{-4}$& 3.4 $\times$ $10^{-4}$ \\
$2\sqrt{2}\pi$ & $24\ \times\ 24$ & 160 & 25 & ref & 4.2 $\times$ $10^{-5}$ \\
\hline
$2\sqrt{2}\pi\pm 10^{-6}$ & $24\ \times\ 24$ & 80 & 25 & 2.5 $\times$ $10^{-4}$& 3.5 $\times$ $10^{-4}$ \\
\hline
\hline
$4\pi$ & $32\ \times\ 32$ & 30 & 28 & 1.2 $\times$ $10^{-1}$& 4.5 $\times$ $10^{-2}$ \\
$4 \pi$ & $32\ \times\ 32$ & 60 & 28  & 2.7 $\times$ $10^{-3}$& 1.6 $\times$ $10^{-3}$ \\
$4 \pi$ & $32\ \times\ 32$ & 180 & 28  & ref & 1.1 $\times$ $10^{-4}$ \\
\hline
\end{tabular}
\caption{\label{errors4} Convergence, as $A$ grows, of the $p=3$
  Neumann solver at and around Wood frequencies. Shift parameter
  $d=1.4$, GMRES residual tolerance equal to $10^{-6}$.  ``Ref" refers to finely resolved solutions against which the error of the coarser solutions is evaluated.}
\end{center}
\end{table}

\section{Numerical results}\label{sec:numerical}
We present numerical computations of scattering by the doubly periodic
scattering surface $f(x,y)=\frac{1}{2}\cos(2 \pi x)\cos(2 \pi y)$ with
periodicity lattice vectors $\vv_1=(1,0,0)$ and $\vv_2=(0,1,0)$ and
under Dirichlet and Neumann boundary conditions.  For non-Wood
configurations we utilize the $p=0$ (unshifted) version of the
algorithm described in Section~\ref{sec:evaluation}.  At and around
Wood configurations, on the other hand, we use the $p=3$ version of
that algorithm.  A fully three-dimensional single-core Matlab
implementation of these methods was used, which was neither
accelerated nor optimized; accordingly, our numerical error studies at
Wood anomalies do not go beyond relative errors of the order
$10^{-4}$. Clear high-order convergence is observed in all cases.

We report the quality of the solutions on the basis of two error
indicators. The first of these indicators is the energy-conservation
defect
\begin{equation}
  \varepsilon=\left|\sum_{(j,\ell)\in P}\frac{\gamma_{j\ell}}{\gamma_{00}}|B_{j\ell}|^2-1\right|
\end{equation}
which we have verified (by means of numerical resolution studies) to
be an excellent error predictor for these solvers.  An additional
error estimator we present, $\varepsilon_1$, on the other hand, equals
the absolute error in the Rayleigh coefficient $B^+_{0,0}$ (as
estimated by comparison with a reference solution obtained by means of
a highly-refined discretization, a large value of the window parameter
$A$, and a sufficiently small GMRES tolerance). The word ``ref'' on a
table entry indicates that the parameter values on that row were used
to produce the reference solution necessary for evaluation of the
errors $\varepsilon_1$ for the corresponding frequency $k$ on the that
table. The numbers of iterations required by the GMRES solvers to
reach specified tolerances are provided in each case.

Table~\ref{errors1} demonstrates the high-order character, as the
window-size parameter $A$ grows, for the proposed $p=0$ (unshifted)
Dirichlet solvers at frequencies $k$ away from Wood anomalies.
Table~\ref{errors11} demonstrates the high-order character
of the $p=3$ (shifted) solver, as $A$ grows, also under
Dirichlet boundary conditions and for values of $k$ away from Wood
anomalies.  Tables~\ref{errors3} and~\ref{errors4}, in turn, concern
configurations at and near Wood anomalies; Dirichlet (resp.  Neumann)
boundary conditions are considered in the first (resp. second) of
these tables. In the normal-incidence case considered in those tables,
the first three Wood anomalies occur at $k=2\pi$, $k=2\sqrt{2}\pi$,
and $k=4\pi$. Once again, fast convergence is observed as $A$ grows,
even at and around Wood anomalies.  We note that the number of
iterations required by the GMRES solvers based on Combined Field
Integral Equations remains small even for Wood and near-Wood
parameters for both Dirichlet and Neumann problems.

\appendix
\section{Appendix: Integral equations away from Wood frequencies}\label{sec:integral}

In order to establish the unique solvability of the integral
equations~\eqref{eq:SLfEPer}, the proof of
Theorem~\ref{thm:uniqueness} relies on the following classical result
on solutions to the homogeneous surface-impedance problem.
\begin{Lemma}\label{lemma:uniqueness}
  Let $v(\bfx)$ is a quasi-periodic field that satisfies the Helmholtz
  equation $\Delta v + k^2 v=0$ for $z<f(\tbfx)$ (resp. $z>f(\tbfx)$), the
  outgoing condition (\ref{outgoing2}) (resp. (\ref{outgoing1})) and
  the impedance condition
\begin{equation*}
   \frac{\partial v}{\partial{n}}(\bfx)-i\zeta\, v(\bfx)=0
  \qquad (\bfx\in\Gamma)
\end{equation*}
with $\zeta>0$ (resp. $\zeta<0$). Then $v(\bfx)=0$ for $z<f(\bfx)$
(resp. $z>f(\bfx)$).
\end{Lemma}

\begin{proof}
  We establish the result in the case $z<f(\bfx)$ and $\zeta>0$; the
  complementary case is handled analogously.  Consider the truncated
  period
\begin{equation*}
  \Omega = \{ (\tbfx,z) : \tbfx \in Q,\,  z_-<z<f(\tbfx) \}
\end{equation*}
with lower boundary $S=\{ (\tbfx,z) : \tbfx \in Q,\, z=z_- \}$
oriented downward.  Using integration by parts we obtain
\begin{equation}\label{parts0}
  0 = \int_\Omega \big(\Delta v + k^2 v\big)\bar v = \int_\Omega \big( \!-\!|\nabla v|^2 + k^2|v|^2 \big)dx
  - \int_{\Gamma^{per}} \frac{\partial v}{\partial n}\,\bar v\ ds- \int_S \frac{\partial v}{\partial n}\,\bar v\ ds\,.
\end{equation}
(The integrals over the lateral sides of $\Omega$ add up to zero as a
result of the assumed quasi-periodicity of $v$.)  In view of the
impedance condition on $\Gamma$ and the outgoing
condition~\eqref{outgoing2} we obtain
\begin{equation*}
  \int_\Gamma \frac{\partial v}{\partial n}\,\bar v\ ds \,+\, \int_S \frac{\partial v}{\partial n}\,\bar v\ ds
  \,=\, i\zeta\int_\Gamma |v|^2\ ds\,+\, \frac{i}{2\dd}\!\!\sum_{(j,\ell), \gamma_{j\ell}>0} \gamma_{j\ell} |c_{j\ell}^{-}|^2.
\end{equation*}
The imaginary part of this quantity equals the imaginary part of
(\ref{parts0}) and it must therefore vanish.  It follows that $v=0$ on
$\Gamma$. The impedance condition then yields $\partial v/\partial{n}
= 0$ on $\Gamma$ as well.  By Green's identity it follows that $v=0$
in $\Omega$ and therefore for all $\bfx$ with $z<f(\tbfx)$
($\tbfx\in\mathbb{R}^2$).
\end{proof}

For completeness we now present a simpler alternative proof of
Theorem~\ref{thm:uniqueness}, which, however, is restricted to the
case $p=0$ (unshifted Green function) and to configurations away from
Wood anomalies.
\begin{Theorem} Let $\frac{\eta}{\xi}<0$, let $p=0$, and let us assume
  that $k$ is a wavenumber for which the quasi-periodic Green function
  $G^q_0$ exists, that is, $\gamma_{j\ell}\not=0$ for all pairs
  $(j,\ell)\in\ZZ$. Then the integral equation~\eqref{eq:SLfEPer} is
  uniquely solvable in $L^2(\Gamma^{per})$.
\end{Theorem}
\begin{proof}
  Given that the surface $\Gamma$ is smooth and that $G^{q}_0$ has the
  same singularity as $G$, it follows that the integral operators on
  the left-hand side of equation~\eqref{eq:SLfEPer} are compact
  operators in the space $L^2(\Gamma^{per})$. Thus, by Fredholm
  theory, the unique solvability of equation~\eqref{eq:SLfEPer} is
  equivalent to the injectivity of the operator on the left-hand side
  of that equation. In order to establish injectivity, let
  $\phi_0\in L^2(\Gamma^{per})$ denote a solution of
  equation~\eqref{eq:SLfEPer} with zero right hand-side, and let
  $u^{\pm}$ denote the restrictions to $\Omega^{\pm}$ of the
  potentials $u$ defined by equation~\eqref{combinedlayer} above and
  below $\Gamma$ with density $\phi = \phi_0$. It follows that $u^+$
  is a radiating solution of the Helmholtz equation in $\Omega^+$ with
  zero Dirichlet boundary conditions, and hence $u^+=0$ in
  $\Omega^+$~\cite{Petit}. Using the jump relations satisfied by the
  layer potentials in equation~\eqref{combinedlayer} we obtain
  $u^{-}|_\Gamma=-\xi\phi_0$ and
  $\left(\frac{\partial
      u^{-}}{\partial{n}}\right)|_\Gamma=i\eta\phi_0$. Thus $u^{-}$ is
  a quasi-periodic radiating solution of the Helmholtz equation in
  $\Omega^{-}$ with zero impedance boundary conditions
  $\frac{\partial v}{\partial{n}}(\bfx)-i\zeta\, v(\bfx)=0,\
  \zeta=-\frac{\eta}{\xi}>0$ on $\Gamma$.  By
  Lemma~\ref{lemma:uniqueness} it follows that $v^-=0$ in
  $\Omega^{-}$, and thus $\phi_0=0$ in $L^2(\Gamma^{per})$, as
  desired. The proof of the theorem is now complete.
\end{proof}

No uniqueness results exist for the Helmholtz scattering
problem~\eqref{helmholtz} under Neumann boundary-value conditions and
the radiation condition~\eqref{outgoing1}, even away from Wood
anomalies---although it has been repeatedly conjectured
(cf.~\cite[p. 147]{KirschUniqueness}),~\cite{BrunoDelourme}) that such a
uniquess result does hold. Assuming that the wavenumber $k$ is such
that this scattering problem does admit a unique solution, however, we
may seek the scattered field in the form
\begin{equation}\label{eq:SLfN}
u(\mathbf x)=\int_{\Gamma} G^q(\mathbf x-\mathbf x')\psi(\mathbf x')ds(\mathbf x'),\quad \mathbf{x}\in\mathbb{R}^3\setminus \Gamma
\end{equation}
in terms of the unknown surface density $\psi$. Using the jump
condition for the normal derivatives of single-layer potentials and
the sound-hard (Neumann) boundary condition, the unknown density
$\psi$ is seen to be a solution of the integral equation
\begin{equation}\label{eq:SLfEN}
-\frac{\psi(\mathbf x)}{2}+\int_{\Gamma}\frac{\partial G^q(\mathbf x-\mathbf x')}{\partial{n}(\mathbf{x})}\psi(\mathbf x')ds(\mathbf x')
=- i(\boldsymbol{\alpha},-\gamma)\cdot\mathbf{n}(\mathbf{x})\ e^{i(\boldsymbol{\alpha}\cdot\tbfx-\gamma z)},\;\; \bfx=(\tbfx,z)\in\Gamma.
\end{equation}
Given that, addtionally, the scattering problems from doubly periodic
surfaces and Dirichlet boundary conditions admit unique solutions for
all wavenumbers, it follows that the integral
equations~\eqref{eq:SLfEN} have themselves unique solutions.

%%%%%%%%%%%%%%%%%%%%%%%%%%%%%%%%%%%%%%%%%%%%%%%%%%

\bigskip
\bigskip

\noindent
{\bfseries\large Data accessibility.}
All data applicable to this paper is included in the article.

\medskip
\noindent
{\bfseries\large Competing interests.}
There are no competing interests relevant to this article.

\medskip
\noindent
{\bfseries\large AuthorsÕ contributions.}
All authors are equally considered co-contributors in this article.

\medskip
\noindent
{\bfseries\large Funding statement.}
This effort was supported by AFOSR, NSF and a NSSEFF Vannevar Bush
Fellowship under contracts FA9550-15-1-0043, DMS-1411876 and
N00014-16-1-2808 (OB); NSF DMS-0807325 (SPS); NSF DMS-1312169 and
DMS-1614270 (CT); and NSF DMS-1211638~(SV).

\medskip
\noindent
{\bfseries\large Ethics statement.}
No ethics statement applies to this article.

%\newpage

\end{document}